\documentclass[preprint,12pt]{elsarticle}
\usepackage{amsmath}
\usepackage{amscd}
\usepackage{amsfonts}
\usepackage{amssymb}
\usepackage{amsthm}
\usepackage{bm}                                                %
\usepackage{float}
\usepackage{graphicx}
\usepackage{latexsym}
\usepackage{lscape}
\usepackage{mathrsfs}
\usepackage{color,xcolor}
\usepackage{verbatim}
\usepackage{framed}
\usepackage{subfigure}

\usepackage[papersize={321mm,369mm}]{geometry}

\usepackage{tikz}
\usetikzlibrary{patterns, positioning, shapes.geometric, arrows.meta}
\usepackage{booktabs}
\usepackage{multirow}
\usepackage{graphicx}
\usepackage{subcaption}

\newtheorem{theorem}{Theorem}[section]
\newtheorem{lemma}{Lemma}[section]

\numberwithin{equation}{section}

\usepackage[notcite,notref]{showkeys}

\begin{document}
	\numberwithin{equation}{section}
	\title{A novel mixed spectral method with ball polynomials for the Biharmonic equation on a unit ball \tnoteref{thank1}}
	\tnotetext[thank1]{This work is partly supported by NSFC grants (12271233), NSF of Shandong Province (ZR2019YQ05).}
	
	\author{Mengxue Gao}\ead{gaomengxue8@gmail.com}
	\author{Bing Su\corref{cor}}\ead{subing@lyu.edu.cn}
	\author{Jianwei Zhou\corref{cor}}\ead{jwzhou@yahoo.com}
	
	\address{School of Mathematics and Statistics, Linyi University, Linyi 276005, P.R. China}
	
	\cortext[cor]{Corresponding author.}
	
	\begin{abstract} 
		A novel mixed spectral-Galerkin method based on generalized ball polynomials is proposed for solving the biharmonic equation on a unit ball. By introducing an auxiliary variable to decouple the biharmonic equation into a system of second-order equations, the corresponding discrete scheme yields a strictly diagonal stiffness matrix, which significantly enhances the computational efficiency. Rigorous a-priori error estimates are established to demonstrate the exponential convergence rates in both the $L^2$- and $H^1$-norms. Extensive numerical experiments are conducted to verify the theoretical analysis and confirm the high efficiency and accuracy of the proposed scheme.
	\end{abstract}
	
	\begin{keyword}  Biharmonic equation, mixed spectral method, ball polynomial, a-priori error estimate.
	\end{keyword}
	
	
	\date{}
	
	\maketitle
	

	\section{Introduction}
	The biharmonic equation plays a fundamental role in various fields of mathematical physics and engineering. Over the past few decades, extensive research has been devoted to developing efficient and accurate numerical schemes for this typical fourth-order partial differential equation. Conventionally, finite element methods and finite difference schemes \cite{EG75} have been widely employed to solve such problems. In recent years, various stabilized and efficient mixed methods have been proposed to improve computational stability and convergence. To alleviate the stringent $C^1$-continuity requirement of standard conforming elements, the mixed finite element method was introduced and subsequently analyzed in \cite{CR74,Mon87}. For more details, we refer the readers to \cite{BG11,KBW23,Lam11,LZ17} and the references cited therein. 
	On the other hand, spectral methods are widely adopted in scientific computing for their capability to achieve exponential convergence rates for biharmonic equations. 
	While extensive studies exist for rectangular or polygonal domains, employing Legendre and Chebyshev polynomials\cite{BK01,She95}, spectral collocation methods\cite{MT07}, as well as Legendre spectral-element approaches \cite{ZC17}, research on canonical domains such as the three-dimensional unit ball requires special consideration.
	The theoretical foundation for using orthogonal polynomials of several variables was established in \cite{DX14}, with subsequent spectral approximation theories on the unit ball in \cite{LX14}. However, efficient direct spectral solvers and rigorous a-priori error analysis for the biharmonic equations remain largely unexplored compared to their second-order counterparts.
	
	In this paper, we propose and analyze a novel mixed spectral-Galerkin method for solving the biharmonic equation with simply supported boundary conditions:
	\begin{equation}\label{eq:biharmonic_model}
		\left\{
		\begin{aligned}
			\Delta^2 u &= f, \quad  \text{in}\ \mathbb{B}^3, \\
			u & = 0, \quad  \text{on}\ \mathbb{S}^2,\\
			\Delta u & = 0, \quad \text{on}\ \mathbb{S}^2,
		\end{aligned}
		\right.
	\end{equation}
	where $\mathbb{B}^3$ denotes the three-dimensional unit ball and $\mathbb{S}^2$ denotes the unit sphere of $\mathbb{B}^3$. 
	We construct an efficient direct spectral-Galerkin scheme by utilizing generalized ball polynomials as basis functions. A significant advantage of this approach is that it is directly formulated on the ball, naturally avoiding the coordinate singularities associated with standard spherical transformations. By introducing an auxiliary variable, the discrete scheme yields a strictly diagonal stiffness matrix, which greatly reduces the computational complexity. Beyond the solver construction, a primary focus of this work is the establishment of a rigorous theoretical framework for the a-priori error analysis in both $L^2$- and $H^1$-norms. 
	Furthermore, we conduct extensive numerical experiments to validate our theoretical results and demonstrate the efficiency of the proposed scheme. We demonstrate that the numerical method achieves spectral accuracy for analytic and transcendental smooth solutions. The numerical results highly align with our theoretical error estimates, confirming the robustness and high precision of the mixed spectral formulation for biharmonic problems on the unit ball.
	
	The remainder of this paper is organized as follows. In Section 2, we introduce the necessary notations and fundamental properties of Jacobi polynomials and generalized ball polynomials. Section 3 details the mixed spectral-Galerkin formulation and rigorously derives the corresponding a-priori error estimates in both the $L^2$- and $H^1$-norms. Section 4 provides numerical experiments to validate the theoretical results and demonstrate the efficiency of our proposed method. Finally, we draw a short conclusion in Section 5.
	
	\section{Preliminaries and Notations} 
	
	We first introduce some standard notations that will be used throughout this paper. Let $\mathbb{N}$ and $\mathbb{N}_0$ denote the set of positive and nonnegative integers, respectively. For any vector $x \in \mathbb{R}^3$, we write $x = (x_1, x_2, x_3)^{\top}$ as a column vector, where $(\cdot)^{\top}$ denotes the transpose operation.
	Let $\mathbb{B}^3 = \{x \in \mathbb{R}^3: \|x\| \leq 1\}$ be unit ball centered at the origin. 
	We adopt the standard Sobolev space $H^m(\mathbb{B}^3)$ with $m \in \mathbb{N}_0$ on $\mathbb{B}^3$, which is equipped with the norm $\|\cdot\|_{m,\mathbb{B}^3}$ and semi-norm $|\cdot|_{m,\mathbb{B}^3}$.  
	The space $H_{0}^m(\mathbb{B}^3)$ is defined as the closure of $C_0^\infty(\mathbb{B}^3)$ within $H^m(\mathbb{B}^3)$. Specially, $H^0(\mathbb{B}^3) = L^2(\mathbb{B}^3)$, and the corresponding inner-product reads:  
	\begin{equation} \label{eq:inner_form}
		(u, v)_{L^2(\mathbb{B}^3)} = \int_{\mathbb{B}^3} u(x) v(x) \, \mathrm{d}x,
	\end{equation}
	and the corresponding norm is given by $\|u\| = \sqrt{(u, u)_{L^2(\mathbb{B}^3)}}$.
	
	Throughout this paper, we use $C$ and $c$ to denote generic positive constants that are independent of the polynomial degree and the domain, which may take different values at different occurrences.

	\subsection{Jacobi polynomials and their Properties} 
	
	Let $I = (-1, 1)$ and $w^{\alpha, \beta}(t) = (1-t)^\alpha (1+t)^\beta$ be a weight function with $\alpha, \beta > -1$. The Jacobi polynomials $\{P_n^{\alpha, \beta}(t)\}_{n \ge 0}$ form an orthogonal basis set for $L^2_{w^{\alpha, \beta}}(I)$ \cite{DX14}, which satisfy the following orthogonality identity:
	\begin{equation} \label{eq:Jacobi-orth} 
		\int_I P_n^{\alpha, \beta}(t) P_m^{\alpha, \beta}(t) w^{\alpha, \beta}(t) \, \mathrm{d}t = \hat{h}_n^{\alpha, \beta} \delta_{nm}, 
	\end{equation} 
	where  
	\begin{equation*} \label{eq:Jacobi-norm} 
		\hat{h}_n^{\alpha, \beta} = \frac{2^{\alpha+\beta+1}}{2n+\alpha+\beta+1} \frac{\Gamma(n+\alpha+1)\Gamma(n+\beta+1)}{n!\Gamma(n+\alpha+\beta+1)}. 
	\end{equation*} 
	
	Note that $P_n^{\alpha, \beta}$ satisfies the Sturm-Liouville equation: 
	\begin{equation*} \label{eq:Jacobi-SL} 
		-\partial_t \left( w^{\alpha+1, \beta+1}(t) \partial_t P_n^{\alpha, \beta}(t) \right) = \lambda_n^{\alpha, \beta} w^{\alpha, \beta}(t) P_n^{\alpha, \beta}(t)
	\end{equation*} 
	with the corresponding eigenvalue $\lambda_n^{\alpha, \beta} = n(n+\alpha+\beta+1)$. 
	Moreover, the derivative of $P_n^{\alpha, \beta}(t)$ relates to a parameter-shifted Jacobi polynomial as follows: 
	\begin{equation*} \label{eq:Jacobi-diff} 
		\partial_t P_n^{\alpha, \beta}(t) = \frac{1}{2}(n+\alpha+\beta+1) P_{n-1}^{\alpha+1, \beta+1}(t). 
	\end{equation*} 
	This index-shifting property is crucial for constructing Sobolev-orthogonal basis functions that are compatible with the mixed weak formulation involving second-order derivatives. 
	
	\subsection{Spherical harmonics and ball polynomials} 
	
	For any non-zero vector ${x} \in \mathbb{B}^3$, we adopt the polar decomposition ${x} = r{\xi}$, where $r = \|{x}\|$ represents the radial distance and ${\xi} = \frac{x}{r} \in \mathbb{S}^2$ denotes the angular direction. The standard $L^2$-inner product on the unit sphere is denoted by 
	\begin{equation*} \label{eq:SH-inner} 
		(u, v)_{\mathbb{S}^2} = \int_{\mathbb{S}^2} u({\xi}) v({\xi}) \, \mathrm{d}\omega({\xi}), 
	\end{equation*} 
	where $\mathrm{d}\omega$ denotes the spherical surface measure. 
	
	We denote by $\mathcal{H}_n(\mathbb{S}^2)$ the finite dimension space with $n$-order spherical harmonics, which is defined as the restriction of homogeneous harmonic polynomials of degree $n$ to the unit sphere \cite{DX14}. The dimension of this space is given by 
	\begin{equation*} \label{eq:dim-Hn} 
		d_n := \dim \mathcal{H}_n(\mathbb{S}^2) = 2n+1. 
	\end{equation*} 
	For each $n \ge 0$, we fix an orthogonal basis set $\{Y_l^n\}_{l=1}^{2n+1}$ for $\mathcal{H}_n(\mathbb{S}^2)$, which satisfies 
	\begin{equation*} \label{eq:SH-orth} 
		(Y_l^n, Y_{l'}^{n'})_{\mathbb{S}^2} = \delta_{nn'} \delta_{ll'}, \quad (n, l), (n', l') \in \Upsilon^3, 
	\end{equation*} 
	where the index sets are defined as
	\begin{equation*} 
		\Upsilon^3 = \left\{(n, l) : 1 \le l \le 2n+1, \, n \in \mathbb{N}_0 \right\}, \quad \Upsilon_N^3 = \left\{(n, l) \in \Upsilon^3 : n \le N \right\}. 
	\end{equation*} 
	
	Recalling the spherical coordinates $x = (\sin\theta \cos\phi, \sin\theta \sin\phi, \cos\theta)^{\top}$ on the unit sphere $\mathbb{S}^2$, i.e., $\|x\|=1$, the orthonormal bases $\{Y_l^n\}_{l=1}^{2n+1}$ can be explicitly constructed by Jacobi polynomials as follows \cite{DX14}:
	\begin{align*}  
		Y_1^n(\theta, \phi) &= \sqrt{\frac{2n+1}{4\pi}} P_n^{0,0}(\cos\theta), \\		
		Y_{2j}^n(\theta, \phi) &= C_{n,j} (\sin\theta)^j \cos(j\phi) P_{n-j}^{j,j}(\cos\theta), \quad	 
		Y_{2j+1}^n(\theta, \phi) = C_{n,j} (\sin\theta)^j \sin(j\phi) P_{n-j}^{j,j}(\cos\theta) ,  
	\end{align*} 
	where $1 \le j \le n$ and the normalization constant
	\begin{equation*} 
		C_{n,j} = \frac{1}{2^{j}} \sqrt{\frac{(2n+1) \Gamma(n-j+1) \Gamma(n+j+1)}{2\pi \Gamma(n+1)^2}}. 
	\end{equation*} 
	
	For $\alpha > -1$, we define the ball polynomials as
	\begin{equation*} \label{eq:ball_poly_def}
		B_{k,l}^{\alpha, n}(x) = \frac{(n+k+\frac{3}{2})_k}{(n+k+\frac{3}{2}+\alpha)_k} P_k^{\alpha, n+\frac{1}{2}}(2\|x\|^2 - 1) Y_l^n(x),
	\end{equation*}
	where $x \in \mathbb{B}^3$, $n$ is the total degree of the spherical harmonic part, and $k \in \mathbb{N}_0$ is the radial degree.
	Note that the total degree of $B_{k,l}^{\alpha, n}(x)$ is $n+2k$.
	
	The ball polynomials are mutually orthogonal with respect to the weight function $\omega_\alpha({x}) = (1 - \|x\|^2)^\alpha$ for $\alpha > -1$ \cite{LX14}:
	\begin{equation*} \label{eq:ball_orth}
		\int_{\mathbb{B}^3} B_{k,l}^{\alpha, n}(x) B_{\iota,\ell}^{\alpha, m}(x) \omega_{\alpha}(x) \, \mathrm{d}x = h_k^{\alpha, n} \delta_{k,\iota} \delta_{l,\ell} \delta_{n,m},
	\end{equation*}
	where the normalization constant is given by
	\begin{equation*}
		h_k^{\alpha, n} = \frac{\Gamma(k+\alpha+1)\Gamma(n+2k+\frac{3}{2}) (n+k+\frac{3}{2})_k}{2 \Gamma(k+1) \Gamma(n+2k+\alpha+\frac{5}{2}) (n+k+\alpha+\frac{3}{2})_k}.
	\end{equation*}
	
	Specifically, we highlight two special cases
	\begin{equation*}
		h_k^{0,n} = \frac{1}{2n+4k+3}, \quad h_k^{1,n} = \frac{2(k+1)(2n+2k+3)}{(2n+4k+3)^2(2n+4k+5)}.
	\end{equation*}
	The orthogonal ball polynomials $\{B_{k,l}^{\alpha, n}(x)\}$ are the eigenfunctions of the following second-order differential operator:
	\begin{equation*}
		-(1-\|x\|^2)^{-\alpha}\nabla \cdot \left[(\mathrm{I}-xx^{\mathrm{t}})(1-\|x\|^2)^{\alpha}\nabla  B_{k,l}^{\alpha, n}(x)\right] = (n+2k)(n+2k+2\alpha+3) B_{k,l}^{\alpha, n}(x),
	\end{equation*}
	where $\rm{I}$ denotes the identity operator.
	
	\subsection{Model Problem and Weak Formulation}
	
	In this section, we define the standard bilinear form
	\begin{equation} \label{eq:bilinear_form} 
		a(u,v) = \int_{\mathbb{B}^3} \nabla u \cdot \nabla v \, \mathrm{d}x.
	\end{equation}
	
	By introducing an auxiliary variable $\sigma = -\Delta u$, we propose the following mixed weak formulation of \eqref{eq:biharmonic_model}: find $(\sigma, u) \in H^1_0(\mathbb{B}^3) \times H^1_0(\mathbb{B}^3)$ such that
	\begin{equation} \label{eq:weak_mixed}
		\begin{cases}
			a(\sigma, v) = (f, v)_{L^2(\mathbb{B}^3)}, & \forall v \in H^1_0(\mathbb{B}^3), \\
			a(u, \tau) = (\sigma, \tau)_{L^2(\mathbb{B}^3)}, & \forall \tau \in H^1_0(\mathbb{B}^3).
		\end{cases}
	\end{equation}
	This splitting framework is standard for the simply supported biharmonic equation \cite{CR74,She95}. The well-posedness of \eqref{eq:weak_mixed} is guaranteed by the Lax-Milgram theorem applied sequentially to the two symmetric positive-definite subproblems. This relies on the continuity and coercivity of $a(\cdot, \cdot)$ on $H^1_0(\mathbb{B}^3)$, i.e., there hold:
	\begin{align}
		|a(u,v)| \le C \|u\|_{H^1(\mathbb{B}^3)} \|v\|_{H^1(\mathbb{B}^3)}, \quad
		a(v,v) \ge c \|v\|_{H^1(\mathbb{B}^3)}^2, \quad \forall u,v \in H^1_0(\mathbb{B}^3). \label{eq:coercivity}
	\end{align} 
	For comprehensive details on these fundamental properties and the associated variational theory, we refer the reader to classical monographs such as \cite{Cia02}.
	
	\section{Novel Spectral-Galerkin Approximation} 
	
	To enforce the homogeneous boundary conditions, we employ generalized ball polynomials to construct an approximation space \cite{DX14, LX14}. 
	The generalized ball polynomials with $\alpha=-1$, which ensure the gradient orthogonality of the basis functions, as follows:
	\[
	B_{0,l}^{-1,n}(x) := B_{0,l}^{0,n}(x), \qquad
	B_{k,l}^{-1,n}(x) := B_{k,l}^{0,n}(x) - B_{k-1,l}^{0,n}(x), \quad k \in \mathbb{N},\quad (n,l)\in \Upsilon^3.
	\]
	Here, $k \in \mathbb{N}$ represents the radial degree, which guarantees that the basis functions vanish on the boundary $\mathbb{S}^2$. 
	
	\subsection{Basis function and mixed formulation} 
	
	Hence, we define finite-dimensional approximation spaces for our mixed spectral-Galerkin method. Let $V_N \subset H^1_{0}(\mathbb{B}^3)$ be the space spanned by basis polynomials of total degree at most $N$: 
	\begin{equation} \label{eq:VN} 
		V_N := \text{span} \left\{ B_{k,l}^{-1, n}({x}) : 2 \le 2k+n \le N, \, k \in \mathbb{N}, \, (n, l) \in \Upsilon_N^3 \right\}. 
	\end{equation} 
	
	A fundamental advantage of this approximation space is that it completely diagonalizes the stiffness matrix \cite{LX14}, whose entries are defined by \eqref{eq:bilinear_form}. In fact, the following gradient orthogonality relation holds: 
	\begin{equation}\label{grad-orth} 
		a(B_{k,l}^{-1,n}, B_{k',l'}^{-1,n'}) = 
		\left( \nabla B_{k,l}^{-1,n}, \nabla B_{k',l'}^{-1,n'} \right)_{L^2(\mathbb{B}^3)} = \lambda_{k,n} \, \delta_{k,k'} \, \delta_{n,n'} \, \delta_{l,l'}, 
	\end{equation} 
	where $$\lambda_{k,n} = 2n + 4k + 1. $$
	
	Using the approximation space defined in \eqref{eq:VN}, we now state the discrete mixed formulation of \eqref{eq:weak_mixed}: find  $(\sigma_N, u_N) \in V_N \times V_N$ such that: 
	\begin{equation} \label{eq:discrete_problem} 
		\begin{cases}
			a(\sigma_N,v_N) = (f,v_N)_{L^2(\mathbb{B}^3)},
			& \forall v_N \in  V_N,\\
			a(u_N,\tau_N) = (\sigma_N,\tau_N)_{L^2(\mathbb{B}^3)},
			& \forall \tau_N \in  V_N.
		\end{cases}
	\end{equation} 
	It is straightforward to conclude that the discrete system \eqref{eq:discrete_problem} is well-posed, which is guaranteed by the judicious choice of the approximation spaces. For more details, we refer the readers to \cite{Cia02, She95}.
	
	\subsection{A-priori error estimates}
	
	In this subsection, we investigate the a-priori error estimates for $u$ and $\sigma$ in our proposed mixed spectral-Galerkin formulation \eqref{eq:discrete_problem}, respectively. 
	To facilitate our analysis, we recall a standard orthogonal projection (i.e., the Ritz projection), $\Pi_N: H^1_0(\mathbb{B}^3) \mapsto V_N$, satisfying
	\begin{equation}\label{grad-dis}
		a(v - \Pi_N v,  w_N) = 0, \quad \forall w_N \in V_N.
	\end{equation}
	
	\begin{lemma}[\cite{DX14, LX14}]
		\label{lem:projection}
		Let $v \in 
		H^m(\mathbb{B}^3) \cap H_0^1(\mathbb{B}^3)$ with $m \ge 1$. Then there holds
		\begin{equation*}
			\|v-\Pi_N v\| \le C N^{-m} \|v\|_{H^m(\mathbb{B}^3)}.
		\end{equation*}
		Moreover, for the $H^1$-norm, we have
		\begin{equation}
			\label{projection}
			\|v-\Pi_N v\|_{H^1(\mathbb{B}^3)} \le C N^{1-m} \|v\|_{H^m(\mathbb{B}^3)}.
		\end{equation}
	\end{lemma}		
	
	Subsequently, we derive a-priori error estimates in $L^2$- and $H^1$-norms by utilizing the aforementioned approximation properties, respectively.
	
	\begin{theorem}\label{thm:quasi_opt_sigma}
		Let $(u,\sigma)$ and $(u_N, \sigma_N)$ be the solutions of \eqref{eq:weak_mixed} and the corresponding discrete system \eqref{eq:discrete_problem}, respectively. For $\sigma \in H^{s-2}(\mathbb{B}^3)$ with $s \ge 3$, there hold the following a-priori error estimates:
		\begin{equation}\label{eq:error-sig1}
			\|\sigma-\sigma_N\|_{H^1(\mathbb{B}^3)}
			\le C N^{3-s}\|\sigma\|_{H^{s-2}(\mathbb{B}^3)}, \quad 
			\|\sigma - \sigma_N\| \le C N^{2-s} \|\sigma\|_{H^{s-2}(\mathbb{B}^3)}.
		\end{equation}
	\end{theorem}
	
	\begin{proof}
		
		For any $v_N \in V_N$, we have
		\begin{align*}
			c \|\sigma - \sigma_N\|_{H^1(\mathbb{B}^3)}^2 
			&\le a(\sigma - \sigma_N, \sigma - \sigma_N) 
			= a(\sigma - \sigma_N, \sigma - v_N) + a(\sigma - \sigma_N, v_N - \sigma_N)\\
			&
			= a(\sigma - \sigma_N, \sigma - v_N)\le C\|\sigma - \sigma_N\|_{H^1(\mathbb{B}^3)} \|\sigma - v_N\|_{H^1(\mathbb{B}^3)},
		\end{align*}	
		which depends on the discrete coercivity and the Galerkin orthogonality \eqref{grad-dis}.
		Consequently, we readily get
		\begin{align*}
			\|\sigma - \sigma_N\|_{H^1(\mathbb{B}^3)}\le C \|\sigma - v_N\|_{H^1(\mathbb{B}^3)}.
		\end{align*}	     	
		Setting $v_N = \Pi_N \sigma$ and applying the approximation \eqref{projection}, we obtain
		\[
		\|\sigma - \sigma_N\|_{H^1(\mathbb{B}^3)} \le C \|\sigma - \Pi_N \sigma\|_{H^1(\mathbb{B}^3)} \le C N^{1-(s-2)} \|\sigma\|_{H^{s-2}(\mathbb{B}^3)} = C N^{3-s} \|\sigma\|_{H^{s-2}(\mathbb{B}^3)}.
		\]
		
		Furthermore, we establish the error estimate in the $L^2$-norm via the standard Aubin-Nitsche duality argument. The dual problem reads: find $\phi\in H^1_{0}(\mathbb{B}^3)\cap H^2(\mathbb{B}^3)$ such that
		\[
		- \Delta \phi = \sigma - \sigma_N,
		\]
		which is equivalent to
		\begin{equation}\label{eq:dual_weak}
			a(\phi,v)=(\sigma - \sigma_N, v)_{L^2(\mathbb{B}^3)}, \quad \forall v\in H^1_{0}(\mathbb{B}^3).
		\end{equation}
		
		Setting $v = \sigma - \sigma_N \in H_0^1(\mathbb{B}^3)$ in \eqref{eq:dual_weak} yields
		\begin{equation*} 
			\begin{aligned}
				\|\sigma - \sigma_N\|^2 = a(\phi, \sigma - \sigma_Ne_\phi) 
				\le  \|\nabla (\sigma - \sigma_N)\| \|\nabla(\phi - \Pi_N\phi)\|
				\le C N^{-1} \|\sigma - \sigma_N\|_{H^1(\mathbb{B}^3)} \|\phi\|_{H^2(\mathbb{B}^3)}.
			\end{aligned}
		\end{equation*}			
		Noting that
		\begin{equation*}\label{eq:adj_reg_sigma}
			\|\phi\|_{H^2(\mathbb{B}^3)}\le C \|\sigma - \sigma_N\|,
		\end{equation*}
		then there holds
		\begin{equation*}\label{eq:adj_sigma}
			\|\sigma - \sigma_N\| 
			\le c N^{-1}\|\sigma - \sigma_N\|_{H^1(\mathbb{B}^3)}
			\le C N^{2-s}\|\sigma\|_{H^{s-2}(\mathbb{B}^3)}.
		\end{equation*}
		This completes the proof.
	\end{proof}
	
	\begin{theorem} 
		\label{thm:u_H1_error}
		Let $(u,\sigma)$ and $(u_N, \sigma_N)$ be the solutions of \eqref{eq:weak_mixed} and the corresponding discrete system \eqref{eq:discrete_problem}, respectively.
		Assume $u \in H^s(\mathbb{B}^3)$ and $\sigma \in H^{s-2}(\mathbb{B}^3)$ with $s \ge 3$.
		Then there hold the following a-priori error estimates
		\begin{equation*}
			c \|u - u_N\|_{H^1(\mathbb{B}^3)} \le N^{1-s} \|u\|_{H^s(\mathbb{B}^3)} 
			+ N^{2-s} \|\sigma\|_{H^{s-2}(\mathbb{B}^3)}, \quad
			C \|u - u_N\| \le   N^{-s}\|u\|_{H^s(\mathbb{B}^3)} + N^{1-s}  \|\sigma\|_{H^{s-2}(\mathbb{B}^3)}.
		\end{equation*}
	\end{theorem}
	\begin{proof}
		
		For any $v_N \in V_N$, it is direct to conclude that
		\begin{equation}\label{eq:orth-Pi}
			a(\Pi_N u - u_N, v_N) = a(u - u_N, v_N) = (\sigma - \sigma_N, v_N)_{L^2(\mathbb{B}^3)}.
		\end{equation}
		Setting $v_N = \Pi_N u - u_N \in H_0^1(\mathbb{B}^3)$ in \eqref{eq:orth-Pi}, we readily deduce that
		\[
		c \|\Pi_N u - u_N \|_{H^1(\mathbb{B}^3)}^2 \le  (\sigma - \sigma_N, \Pi_N u - u_N )_{L^2(\mathbb{B}^3)} \le \|\sigma - \sigma_N\| \|\Pi_N u - u_N \|.
		\]
		By the Poincar\'{e} inequality, there holds  
		\[
		\|\Pi_N u - u_N \|_{H^1(\mathbb{B}^3)} \le  C \|\sigma - \sigma_N\|.
		\]
		Then we obtain that
		\[
		\|u - u_N\|_{H^1(\mathbb{B}^3)} \le \|u - \Pi_N u\|_{H^1(\mathbb{B}^3)} + \|\Pi_N u  - u_N\|_{H^1(\mathbb{B}^3)} 
		\le C N^{1-s} \|u\|_{H^s(\mathbb{B}^3)} 
		+ C N^{2-s} \|\sigma\|_{H^{s-2}(\mathbb{B}^3)}.
		\]
		Subtracting \eqref{eq:discrete_problem} from \eqref{eq:weak_mixed}, we obtain the following identities:
		\begin{subequations}
			\begin{align}
				a(\sigma - \sigma_N, v_N) &= 0, \quad v_N \in V_N, \label{eq:error_sigma} \\
				a(u - u_N, w_N) &= (\sigma - \sigma_N, w_N)_{L^2(\mathbb{B}^3)}, \quad w_N \in V_N. \label{eq:error_u}
			\end{align}
		\end{subequations}
		
		To derive $L^2$-norm estimate for $u - u_N$, we recall an auxiliary dual problem: find $\psi \in H^4(\mathbb{B}^3)$ such that 
		\begin{equation}\label{eq:dual_biharmonic}
			\left\{
			\begin{aligned}
				\Delta^2 \psi  &= u - u_N, \quad  \text{in}\ \mathbb{B}^3, \\
				\psi & = 0, \quad  \text{on}\ \mathbb{S}^2,\\
				\Delta \psi & = 0, \quad \text{on}\ \mathbb{S}^2.
			\end{aligned}
			\right.
		\end{equation}
		Setting
		$-\Delta \eta = u - u_N$ and $-\Delta \psi = \eta $, the equivalent discrete weak formulations read: find $\eta_N \in V_N$ and $\psi_N \in V_N$ such that
		$$a(\eta_N, v_N) = (u - u_N, v_N)_{L^2(\mathbb{B}^3)}, \quad a(\psi_N, v_N) = (\eta_N, v_N)_{L^2(\mathbb{B}^3)}, \quad \forall v_N\in V_N.$$
		Based on the elliptic regularity, we have
		\begin{equation}\label{eq:dual_regularity}
			\|\eta\|_{H^2(\mathbb{B}^3)} \le C \|u - u_N\|, \quad \|\psi\|_{H^4(\mathbb{B}^3)} \le c \|\eta\|_{H^2(\mathbb{B}^3)} \le C \|u - u_N\|.
		\end{equation}
		It is direct to calculate that
		\begin{equation}\label{eq:L2_rep1}
			\begin{aligned}
				\|u - u_N\|^2 & = (u - u_N, -\Delta \eta)_{L^2(\mathbb{B}^3)} = a(u - u_N, \eta) 
				= a(u - u_N, \eta - \eta_N) + a(u - u_N, \eta_N) \\
				& = a(u - u_N, \eta - \eta_N) + (\sigma - \sigma_N, \eta_N)_{L^2(\mathbb{B}^3)}.
			\end{aligned}
		\end{equation}
		
		Next, for any $\psi_N \in V_N$, we utilize the orthogonality \eqref{eq:error_sigma} to get
		\begin{equation}\label{eq:L2_rep2}
			\begin{aligned}
				(\sigma - \sigma_N, \eta_N)_{L^2(\mathbb{B}^3)} 
				= (\sigma - \sigma_N, \eta_N + \Delta \psi_N)_{L^2(\mathbb{B}^3)} + a(\sigma - \sigma_N, \psi_N) = (\sigma - \sigma_N, \eta_N + \Delta \psi_N)_{L^2(\mathbb{B}^3)}.
			\end{aligned}
		\end{equation}
		Combining \eqref{eq:L2_rep1} with \eqref{eq:L2_rep2}, we arrive at
		\begin{equation*}
			\|u - u_N\|^2 = a(u - u_N, \eta - \eta_N) + (\sigma - \sigma_N, \eta_N + \Delta \psi_N)_{L^2(\mathbb{B}^3)}.
		\end{equation*}
		One directly declares that
		\begin{equation*}
			c \|u - u_N\|^2 \le \|u - u_N\|_{H^1(\mathbb{B}^3)} \|\eta - \Pi_N \eta\|_{H^1(\mathbb{B}^3)} + \|\psi - \psi_N \|_{H^2(\mathbb{B}^3)}
			\|\eta - \Pi_N \eta \|.
		\end{equation*}
		In view of the projection approximation properties in Lemma \ref{lem:projection} and \eqref{eq:dual_regularity}, we have
		\begin{align*}
			\|\eta - \Pi_N \eta\|_{H^1(\mathbb{B}^3)} \le C N^{-1} \|u - u_N\|,  \quad
			\|\eta - \Pi_N \eta\| \le c N^{-2} \|u - u_N\|.
		\end{align*}
		And hence, we obtain
		\begin{equation*}
			c \|u - u_N\| \le  N^{-1} \|u - u_N\|_{H^1(\mathbb{B}^3)} + N^{-2} \|\sigma - \sigma_N\|.
		\end{equation*}
		
		Finally, by incorporating the previously established $H^1$- and $L^2$-error estimates for $u$ and $\sigma$, we readily obtain
		\begin{align*}
			c \|u - u_N\|
			&\le N^{-1} ( N^{1-s} \|u\|_{H^s(\mathbb{B}^3)} + N^{2-s} \|\sigma\|_{H^{s-2}(\mathbb{B}^3)} ) + N^{-2} ( N^{2-s} ) \|\sigma\|_{H^{s-2}(\mathbb{B}^3)}\\
			&\le N^{-s}\|u\|_{H^s(\mathbb{B}^3)} + N^{1-s}  \|\sigma\|_{H^{s-2}(\mathbb{B}^3)},
		\end{align*}
		which completes the proof.
	\end{proof}

	\section{Numerical Experiments}
	
	In this section, we present two numerical examples to validate the accuracy and convergence rates of the proposed mixed spectral-Galerkin method. 
	Throughout this section, the maximum polynomial degree is denoted by $N$.
	To quantify the exponential convergence, we calculate the experimental convergence rate:
	$$\text{Rate}_v=\frac{\ln \left( E_{i-1}(v) / E_i(v) \right)}{N_{i} - N_{i-1}} ,$$
	where $E_i(v)$ denotes the corresponding $L^2$-error of $v$ at $N_i$.
	We employ $r = \sqrt{x_1^2 + x_2^2 + x_3^2}$ to denote the radial distance from the origin.
	
	\subsection{Case 1: Analytic function}
	To verify the  accuracy, we construct an exact solution:
	\begin{equation}\label{eq:exact_u}
		u(r) = \begin{cases}
			\frac{\sin(\pi r)}{r}, & r > 0, \\
			\pi, & r = 0.
		\end{cases}
	\end{equation}
	Note that this function is smooth. We evaluate the proposed mixed spectral-Galerkin method by measuring the numerical errors in both the $L^2$- and the $H^1$-norms. Table \ref{tab:errors} presents the numerical results for a sequence of polynomial degrees $N = 4, 8, 12, 16$.
	
	\begin{table}[H]
		\centering
		\caption{Numerical errors and convergence rates for Case 1}
		\label{tab:errors}
		\renewcommand{\arraystretch}{1.2}
		\begin{tabular}{|c|c|c|c|c|c|c|}
			\hline
			$N$ & $\|\sigma - \sigma_{N}\|_{H^1(\mathbb{B}^3)}$ & $\|u - u_N\|_{H^1(\mathbb{B}^3)}$ & $\|\sigma - \sigma_{N}\|$ & $\|u - u_N\|$ & $\text{Rate}_{\sigma}$ & $\text{Rate}_{u}$ \\
			\hline
			4  & 8.182848e-01 & 9.950927e-02 & 2.259174e-01 & 2.428087e-02 & -- & -- \\
			8  & 1.857498e-03 & 1.949815e-04 & 1.884392e-04 & 1.941540e-05 & 1.7723 & 1.7828 \\
			12 & 6.497892e-07 & 6.664685e-08 & 3.650210e-08 & 3.727417e-09 & 2.1373 & 2.1395 \\
			16 & 6.151159e-11 & 7.033383e-12 & 2.469976e-12 & 2.499720e-13 & 2.4002 & 2.4025 \\
			\hline
		\end{tabular}
	\end{table}
	
	The quantitative results in Table \ref{tab:errors} clearly demonstrate spectral accuracy of our proposed scheme. As the polynomial degree $N$ increases, the errors decay exponentially, which is consistent with the high smoothness of the exact solution. 
%
	To further visualize this convergence behavior, Figure \ref{fig:convergence} presents the semi-logarithmic profiles of the $L^2$-errors for both $u$ and $\sigma$ comparing with a red dashed line segment equipped slope $-1$. The strictly linear trajectory observed in these plots is a hallmark of spectral convergence. This confirms that our mixed formulation effectively captures the smooth features of the solution with high precision and stability.

		\begin{figure}[htbp]
			\centering
			\includegraphics[width=8in]{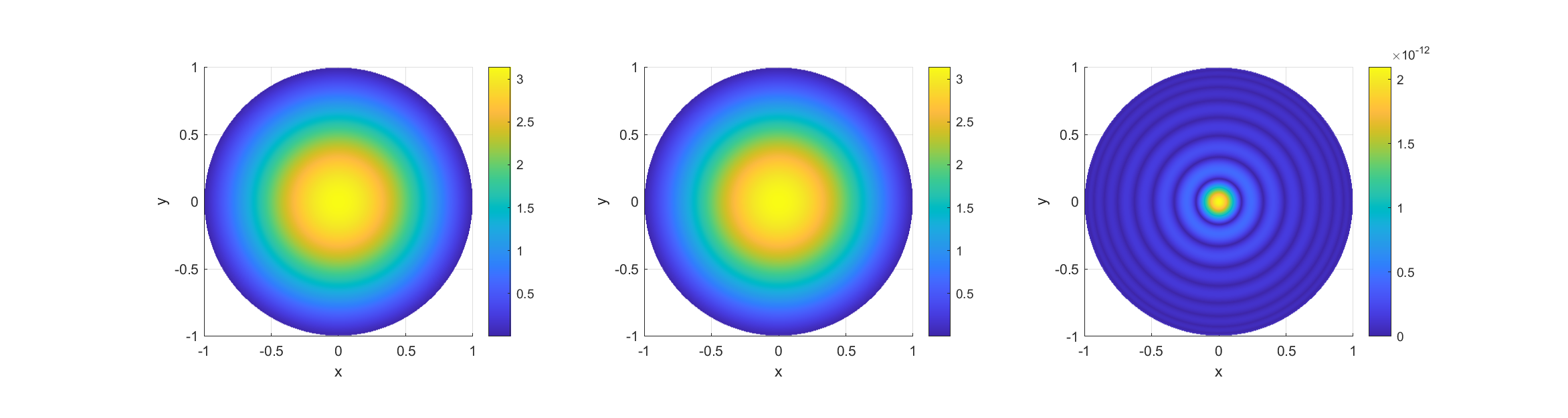}
			\caption{The profiles of $u$, $u_N$ and its point-wise errors at $N=16$.}
			\label{fig:e1_u_16}
		\end{figure}
	
			\begin{figure}[htbp]
				\centering
				\includegraphics[width=8in]{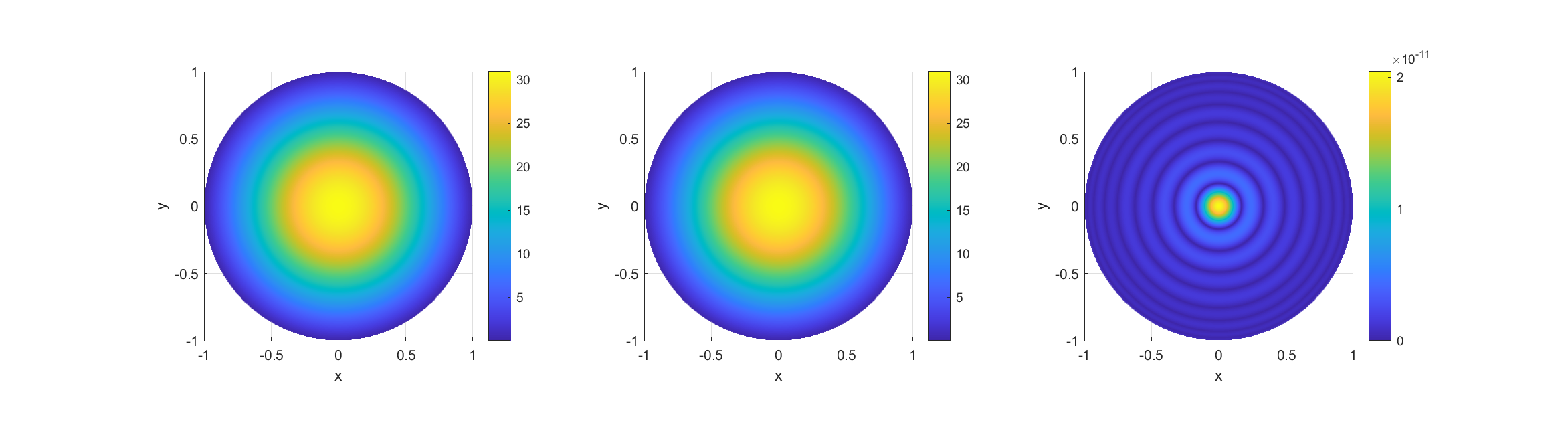}
				\caption{The profiles of $\sigma$, $\sigma_N$ and its point-wise errors at $N=16$.}
				\label{fig:e1_sig_16}
			\end{figure}
	
	
	\begin{figure}[H]
		\centering
		\subfigure[Convergence profiles of $\|u - u_N\|$]{
			\includegraphics[width=3.5in]{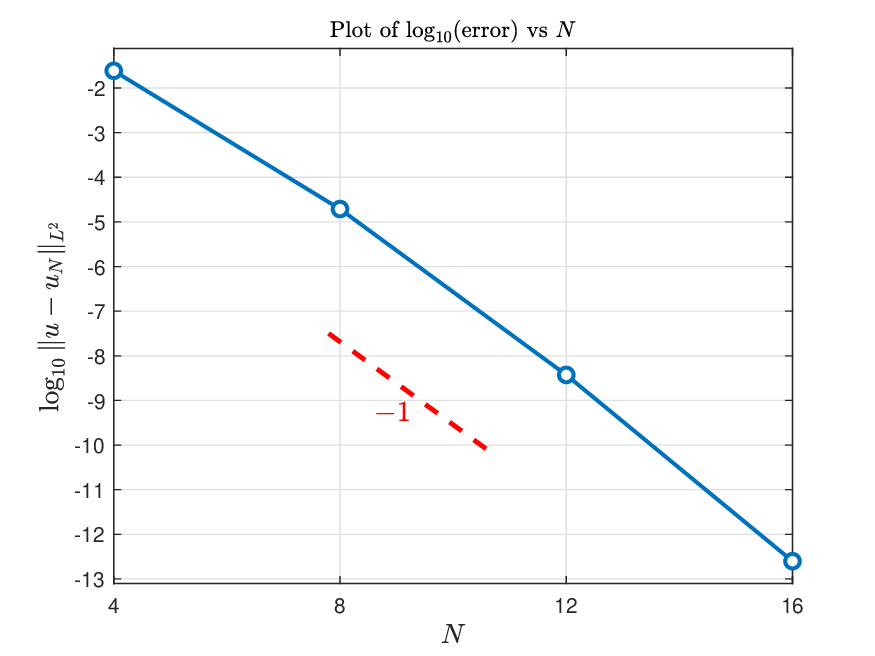}
		}
		\hfill
		\subfigure[Convergence profiles of $\|\sigma - \sigma_{N}\|$]{
			\includegraphics[width=3.5in]{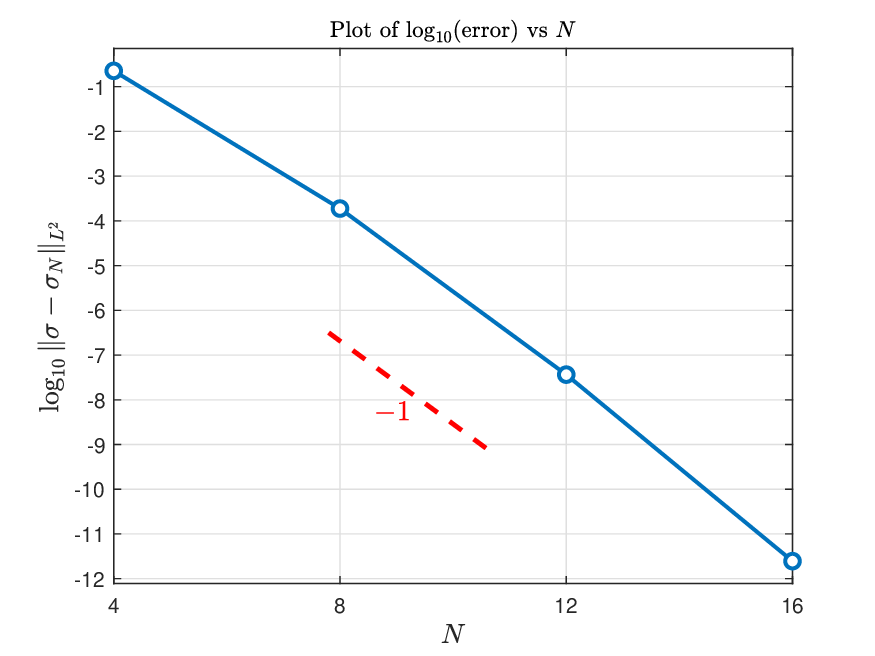}
		}
		\caption{Semi-logarithmic convergence for Case 1}\label{fig:convergence}
	\end{figure}

	\subsection{Case 2: Transcendental function}
	
	To further assess the accuracy of the proposed method for smooth but non-polynomial functions. We construct the exact solution as:
	\begin{equation}
		u(r) = \mathrm{e}^{r^2} - \frac{5\mathrm{e}}{3}r^2 + \frac{2\mathrm{e}}{3}.
	\end{equation}
	This solution is analytic but not a polynomial and exhibits monotonic behavior in the radial direction. Table \ref{tab:smooth_transcendental} presents the convergence results for the primal and auxiliary variables. Consistent with the previous case, the method exhibits excellent spectral accuracy, with the $L^2$-error of $u$ decaying rapidly from $\mathcal{O}(10^{-2})$ at $N=4$ to $\mathcal{O}(10^{-11})$ at $N=16$.
	
	\begin{table}[H]
		\centering
		\caption{Numerical errors and convergence rates for Case 2}
		\label{tab:smooth_transcendental}
		\renewcommand{\arraystretch}{1.2}
		\begin{tabular}{|c|c|c|c|c|c|c|}
			\hline
			$N$ & $\|\sigma - \sigma_{N}\|_{H^1(\mathbb{B}^3)}$ & $\|u - u_N\|_{H^1(\mathbb{B}^3)}$ & $\|\sigma - \sigma_{N}\|$ & $\|u - u_N\|$ & $\text{Rate}_{\sigma}$ & $\text{Rate}_{u}$ \\
			\hline
			4  & 1.222364e+00 & 3.533737e-02 & 3.205838e-01 & 1.417889e-02 & -- & -- \\
			8  & 1.330501e-02 & 4.220701e-04 & 1.309173e-03 & 4.418767e-05 & 1.3752 & 1.4420 \\
			12 & 4.421004e-05 & 1.165379e-06 & 2.431017e-06 & 6.526781e-08 & 1.5722 & 1.6294 \\
			16 & 6.996362e-08 & 1.543859e-09 & 2.536715e-09 & 5.630076e-11 & 1.7163 & 1.7639 \\
			\hline
		\end{tabular}
	\end{table}
	
	The exponential decay of the numerical errors is further illustrated in Figure \ref{fig:case2_plots}, which is compared with a red dashed line segment equipped slope $-1$. The distinct linear trajectories observed in the semi-logarithmic profiles for $u$ and $\sigma$ strongly confirm the spectral convergence. These results reinforce the high precision of the proposed mixed spectral-Galerkin method, even when applied to non-oscillatory transcendental solutions.
	
%

		\begin{figure}[htbp]
			\centering
			\includegraphics[width=8in]{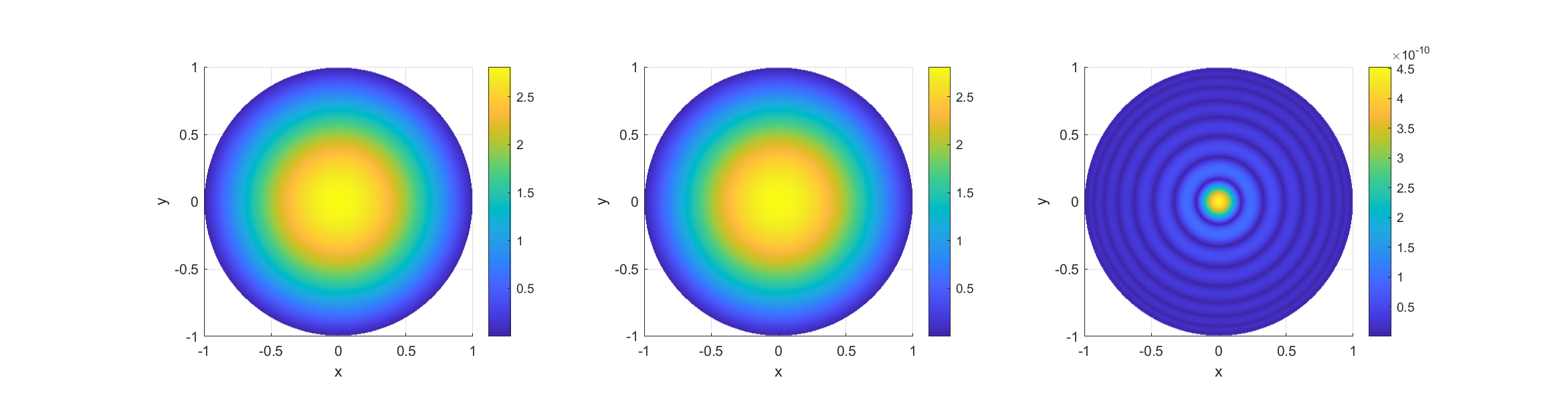}
			\caption{The profiles of $u$, $u_N$ and its point-wise errors at $N=16$}
			\label{fig:e2_u_16}
		\end{figure}
		
			\begin{figure}[htbp]
			\centering
			\includegraphics[width=8in]{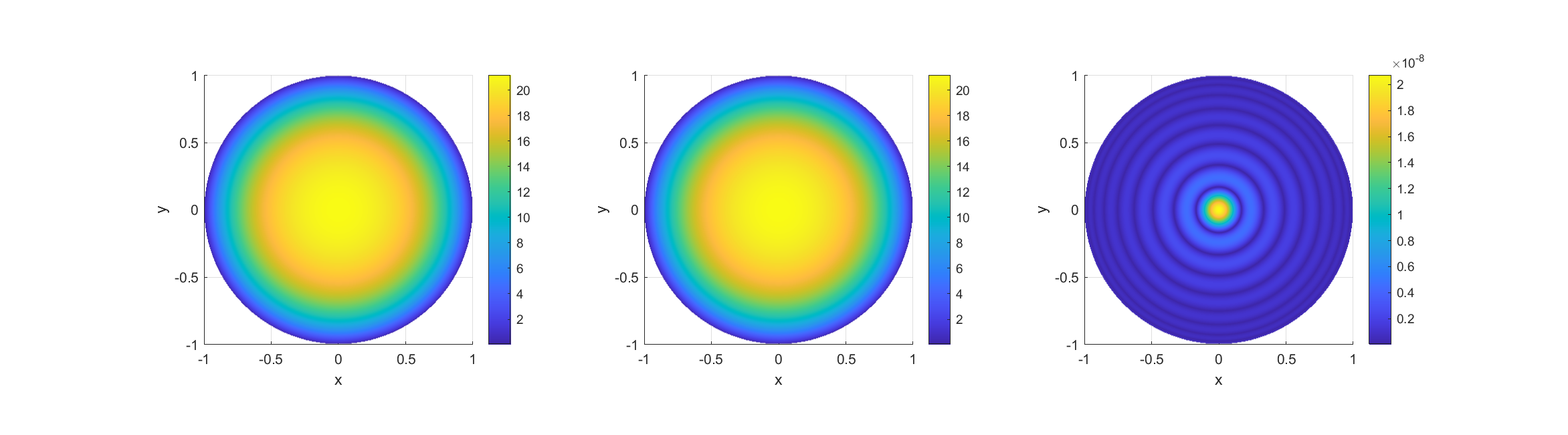}
			\caption{The profiles of $\sigma$, $\sigma_N$ and its point-wise errors at $N=16$}
			\label{fig:e2_sig_16}
		\end{figure}
		
	\begin{figure}[H]
		\centering
		\subfigure[Convergence profiles of $\|u - u_N\|$]{
			\includegraphics[width=3.5in]{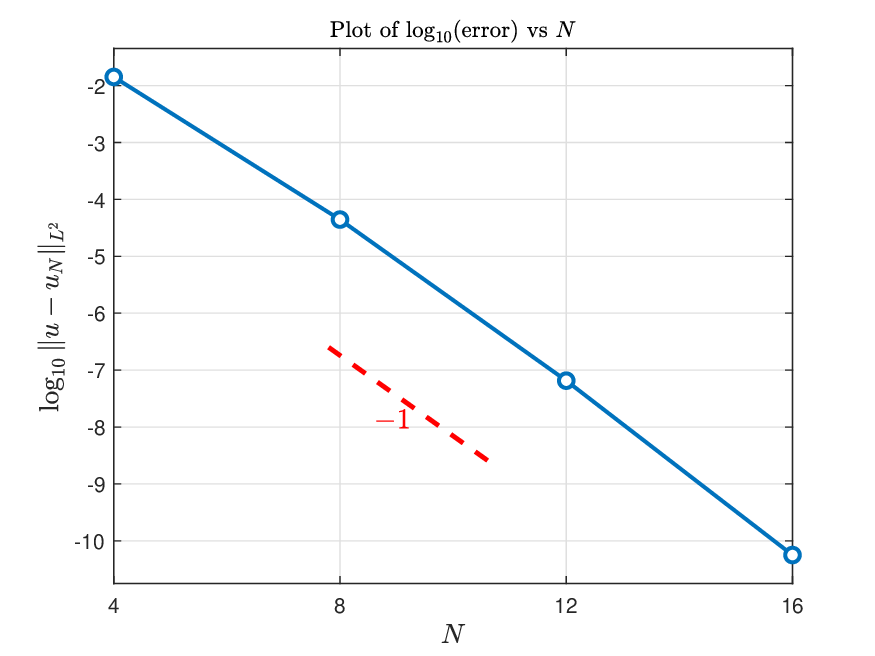}
		}
		\hfill
		\subfigure[Convergence profiles of $\|\sigma - \sigma_{N}\|$]{
			\includegraphics[width=3.5in]{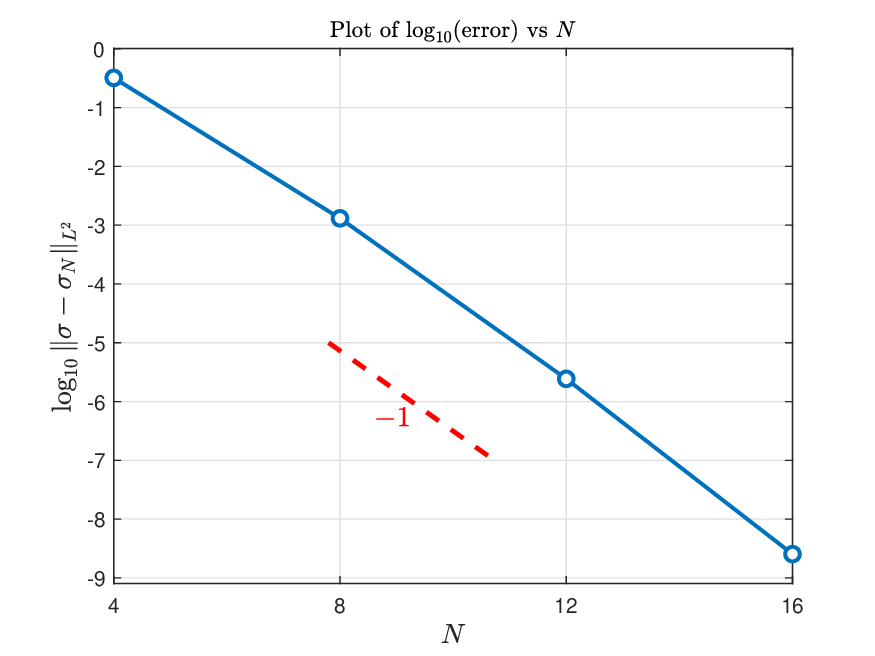}
		}
		\caption{Semi-logarithmic convergence for Case 2}
		\label{fig:case2_plots}
	\end{figure}

	\section{Conclusions}
	
	In this paper, we have proposed and analyzed a mixed spectral-Galerkin method for solving the biharmonic equation with simply supported boundary conditions on the unit ball. By introducing an auxiliary variable, the original fourth-order problem is successfully decoupled into a system of second-order equations. A key feature of our method is the basis construction using generalized ball polynomials, which yields a strictly diagonal stiffness matrix, thereby significantly reducing the computational cost. Furthermore, we rigorously proved the a-priori error estimates in both the $L^2$- and $H^1$-norms for the primal variable and auxiliary variable, respectively. The extensive numerical experiments highly align with our theoretical analysis, further confirming the spectral accuracy and efficiency of the proposed mixed method.

	
	\bibliographystyle{siamplain}

\end{document}